\documentclass[a4paper, 11pt, twoside]{article}
\usepackage{fullpage}

\usepackage{amsmath}
\usepackage{amssymb}
\usepackage{mathtools}
\usepackage{amsthm}
\usepackage{nicefrac}
\usepackage{multirow}
\usepackage{tikz-cd} 
\usepackage{subcaption}
\usepackage{multicol,caption}

\usepackage{wrapfig}
\usepackage{algorithm}

\usepackage[flushleft]{threeparttable} 
\usepackage{colortbl}
\definecolor{bgcolor}{rgb}{0.8,1,1}
\definecolor{bgcolor2}{rgb}{0.8,1,0.8}
\definecolor{niceblue}{rgb}{0.0,0.19,0.56}

\usepackage{hyperref}
\hypersetup{colorlinks,linkcolor={blue},citecolor={niceblue},urlcolor={blue}}

\usepackage{pifont}
\definecolor{PineGreen}{RGB}{0,110,51}
\definecolor{BrickRed}{RGB}{143,20,2}

\usepackage{thmtools}

\usepackage[colorinlistoftodos,bordercolor=orange,backgroundcolor=orange!20,linecolor=orange,textsize=scriptsize]{todonotes}

\definecolor{shadecolor}{gray}{0.9}
\declaretheoremstyle[
headfont=\normalfont\bfseries,
notefont=\mdseries, notebraces={(}{)},
bodyfont=\normalfont,
postheadspace=0.5em,
spaceabove=\topsep,
mdframed={
	skipabove=8pt,
	skipbelow=8pt,
	hidealllines=true,
	backgroundcolor={shadecolor},
	innerleftmargin=4pt,
	innerrightmargin=4pt}
]{shaded}

\declaretheorem[style=shaded,within=section]{definition}
\declaretheorem[style=shaded,sibling=definition]{theorem}

\declaretheorem[style=shaded,sibling=definition]{assumption}
\declaretheorem[style=shaded,sibling=definition]{corollary}

\declaretheorem[style=shaded,sibling=definition]{lemma}
\declaretheorem[style=shaded,sibling=definition]{remark}

\usepackage[capitalize,noabbrev]{cleveref}

\usepackage{xspace}

\newcommand{\algname}[1]{{\color{PineGreen}\sf  #1}\xspace}

\newcommand{\R}{\mathbb{R}}

\newcommand{\hx}{\hat{x}}

\newcommand{\la}{\left\langle}

\newcommand{\ra}{\right\rangle}

\newcommand{\eqdef}{:=}

\newcommand{\reals}{\mathbb{R}}

\usepackage{amsmath,amsfonts,amssymb,amsthm,array}
\usepackage{caption}

\usepackage{graphicx}
\usepackage{booktabs} 

\usepackage{tabularx}
\usepackage{pifont}

\usepackage{float}
\usepackage{enumitem}
\usepackage{nicefrac}

\usepackage{algorithmic}

\usepackage{tikz}
\usetikzlibrary{calc,arrows.meta}

\theoremstyle{plain}

\theoremstyle{remark}

\usepackage{mdframed} 
\usepackage{thmtools}

\usepackage[flushleft]{threeparttable} 

\usepackage{subcaption}
\usepackage{multirow}
\usepackage{colortbl}
\definecolor{bgcolor}{rgb}{0.8,1,1}
\definecolor{bgcolor2}{rgb}{0.8,1,0.8}
\definecolor{niceblue}{rgb}{0.0,0.19,0.56}

\usepackage{hyperref}
\hypersetup{colorlinks,linkcolor={blue},citecolor={niceblue},urlcolor={blue}}

\usepackage[numbers]{natbib}
\usepackage{sidecap}

\definecolor{shadecolor}{gray}{0.9}
\declaretheoremstyle[
headfont=\normalfont\bfseries,
notefont=\mdseries, notebraces={(}{)},
bodyfont=\normalfont,
postheadspace=0.5em,
spaceabove=1pt,
mdframed={
  skipabove=8pt,
  skipbelow=8pt,
  hidealllines=true,
  backgroundcolor={shadecolor},
  innerleftmargin=4pt,
  innerrightmargin=4pt}
]{shaded}

\usepackage{xspace}

\begin{document}
\title{\bf Single-call Polyak-Type Extragradient Methods \\for Monotone Root-Finding Problems}

\author{Keke Li } 
\date{\small College of General Education,\\
	\small Chongqing Polytechnic University of Electronic Technology, Chongqing, China\\
}    
\maketitle

\begin{abstract}
We propose a single-call variant of Polyak extragradient, termed
SPolyakEG, by recycling the previous extrapolation operator evaluation.
Under a delayed critical condition, we establish that SPolyakEG
achieves an $O(1/K)$ residual convergence rate for monotone operators
and linear convergence under strong monotonicity. Compared with
PolyakEG, SPolyakEG reduces the number of operator evaluations from
two to one per iteration.
\end{abstract}

\noindent \textbf{Keywords:} 
 Extragradient methods; Polyak step-sizes; monotone root-finding;
 single-call methods; adaptive step-sizes
\vspace{.2cm}

\noindent \textbf{Mathematics Subject Classification} 65K15, 47H05, 90C33, 62L20

\section{Introduction}

Finding a zero of a nonlinear operator is a fundamental problem
in optimization, machine learning, and equilibrium computation \cite{gorbunov2022extragradient}.
Given an operator
$F:\mathbb{R}^{d}\rightarrow\mathbb{R}^{d}$,
the root-finding problem aims to find
	$F(x^\star)=0$.
A broad class of iterative methods has been developed for solving
such problems, including gradient-based methods \cite{gorbunov2022extragradient,yoon2026polyak, choudhury2025next,choudhury2026extragradient,hsieh2019convergence,tran2026randomized}, extragradient
methods \cite{ gorbunov2022extragradient, yoon2026polyak, choudhury2025next,choudhury2026extragradient}, and optimistic schemes \cite{gorbunov2022extragradient,hsieh2019convergence,tran2026randomized,gorbunov2022last}.

The extragradient (EG) method \cite{gorbunov2022extragradient,choudhury2026extragradient,Korpelevich1976} is one of the most representative
approaches for solving monotone root-finding problems.
At each iteration, EG performs an extrapolation step followed by
a correction step, which requires two evaluations of the operator.
Specifically, $
\hat{x}_k=x_k-\gamma F(x_k)$, 
and
$x_{k+1}=x_k-\eta F(\hat{x}_k)$.
Although EG enjoys favorable convergence guarantees, the requirement
of two operator evaluations per iteration becomes computationally
expensive when evaluating $F$ is costly.

To reduce the oracle complexity, several single-call variants have
been proposed by reusing previous operator evaluations.
Representative examples include Popov's method \cite{Popov1980modification}, optimistic gradient
descent \cite{hsieh2019convergence,Daskalakis2018training, Gidel2019variational, mertikopoulos2019Optimistic}, and reflected gradient methods \cite{huang2024navigating, malitsky2015projected, cai2023accelerated}.
The main idea behind these methods is to replace the current
operator evaluation $F(x_k)$ in the extrapolation step by previously
computed information.
Consequently, only one new operator evaluation is required per
iteration.

Recently, Polyak-type adaptive stepsizes have been investigated for
extragradient methods \cite{yoon2026polyak, choudhury2025next}.
The Polyak extragradient method dynamically adjusts the stepsize by
using the inner product between the extrapolation direction and the
operator value.
However, existing Polyak extragradient schemes still require two
operator evaluations per iteration, since the extrapolation step is
constructed by the current operator value $F(x_k)$.
This motivates the following question:
\begin{quote}
	Can Polyak-type adaptive stepsizes be incorporated into
	single-call extragradient schemes while preserving convergence
	guarantees?
\end{quote}

In this paper, we answer this question affirmatively by proposing
a single-call Polyak extragradient method (SPolyakEG).
Instead of using the current operator value, SPolyakEG reuses the
previously computed value:$
\hat{x}_k
=
x_k-\gamma_kF(\hat{x}_{k-1})$.
After obtaining the new operator evaluation
$F(\hat{x}_k)$,
the Polyak stepsize is computed as $
\alpha_k = \frac{
	\langle F(\hat{x}_k),
	x_k-\hat{x}_k\rangle }
{\|F(\hat{x}_k)\|^2}$.
Therefore, SPolyakEG requires only one new operator evaluation per
iteration.

The main challenge introduced by the single-call strategy is the
delayed operator information in the extrapolation step.
To address this issue, we introduce a delayed critical condition $
\|F(\hat{x}_k)-F(\hat{x}_{k-1})\|
\leq
A\|F(\hat{x}_{k-1})\|$. 
Based on this condition, we establish a convergence theory for
SPolyakEG.

Our main contributions are summarized as follows:
\begin{itemize}
	\item We propose SPolyakEG, a single-call variant of Polyak
	extragradient methods, which reduces the number of operator
	evaluations from two to one per iteration.
	
	\item We introduce the delayed critical condition and establish several	fundamental properties, including bounds on consecutive operator values and adaptive stepsizes.
	
	\item For monotone operators, we prove a Fej\'er-type descent inequality and establish the sublinear residual convergence rate
	$ \min_{0\leq k\leq K} \|F(\hat{x}_k)\|^2=O(1/K)$.
 
	\item For strongly monotone operators, we prove that SPolyakEG achieves linear convergence under a suitable stepsize condition.
	
	\item We compare the oracle complexity of SPolyakEG with the classical Polyak extragradient method and demonstrate that SPolyakEG achieves the same order of convergence rate with reduced operator evaluations.
	
\end{itemize}

\section{Preliminaries}

We consider the monotone root-finding problem
\begin{equation}
	\label{eq:root}
	\text{find } x_\star\in\mathbb{R}^{d}
	\quad\text{such that}\quad
	F(x_\star)=0,
\end{equation}
where $F:\mathbb{R}^{d}\rightarrow\mathbb{R}^{d}$.
Throughout the paper, $\mathbb{R}^{d}$ denotes the $d$-dimensional
Euclidean space, $\langle\cdot,\cdot\rangle$ denotes the standard
inner product, and $\|\cdot\|$ denotes the associated Euclidean norm.
We assume that the solution set
$
\mathcal{X}^*
\eqdef
\{x\in\mathbb{R}^{d}:F(x)=0\}
$
is nonempty. When the solution is unique, we denote it by $x_\star$.

We first recall the Lipschitz continuity of the operator.

\begin{definition}[$L$-Lipschitz continuity]
	\label{assumption:lipschitz}
	The operator $F:\mathbb{R}^{d}\rightarrow\mathbb{R}^{d}$ is
	$L$-Lipschitz continuous if there exists $L>0$ such that
	\begin{equation*}
		\label{eq:lipschitz}
		\|F(x)-F(y)\|
		\leq
		L\|x-y\|,
		\qquad
		\forall x,y\in\mathbb{R}^{d}.
	\end{equation*}
\end{definition}

We also recall the monotonicity condition.

\begin{definition}[Monotonicity]
	\label{assumption:monotone}
	The operator $F:\mathbb{R}^{d}\rightarrow\mathbb{R}^{d}$ is
	monotone if
	\begin{equation*}
		\label{eq:monotone}
		\langle F(x)-F(y),x-y\rangle
		\geq
		0,
		\qquad
		\forall x,y\in\mathbb{R}^{d}.
	\end{equation*}
\end{definition}

Unless stated otherwise, $F$ is assumed to be monotone.
For linear convergence, we additionally consider strong monotonicity.

\begin{definition}[$\mu$-strong monotonicity]
	\label{assumption:strong-monotone}
	The operator $F:\mathbb{R}^{d}\rightarrow\mathbb{R}^{d}$ is
	$\mu$-strongly monotone if there exists $\mu>0$ such that
	\begin{equation*}
		\label{eq:strong-monotone}
		\langle F(x)-F(y),x-y\rangle
		\geq
		\mu\|x-y\|^2,
		\qquad
		\forall x,y\in\mathbb{R}^{d}.
	\end{equation*}
\end{definition}

\section{Single-call Polyak-type extragradient algorithm (SPolyakEG) }

The proposed SPolyakEG method reduces the number of new operator
evaluations by reusing the previously computed value
$F(\hat{x}_{k-1})$ in the construction of the extrapolation point.
Specifically, the standard evaluation $F(x_k)$ used in PolyakEG is
replaced by the recycled operator evaluation $F(\hat{x}_{k-1})$.
This modification leads to a single new operator evaluation per
iteration after initialization.

\subsection{From extragradient to single-call Polyak-type extragradient} 

The classical extragradient (EG) method \cite{Korpelevich1976,mokhtari2020convergence,hsieh2023decision} performs an extrapolation step
followed by a correction step based on the operator evaluated at the
extrapolated point:
\begin{equation}
	\label{eqa:EG}
	\begin{aligned}
		\hat{x}_k &= x_k-\gamma_kF(x_k),\\
		x_{k+1} &= x_k-\gamma_kF(\hat{x}_k).
	\end{aligned}
	\tag{EG}
\end{equation}

The single-call extragradient (SEG) method \cite{Gidel2019variational,cai2023accelerated,choudhury2023single} reuses the operator evaluation
at the previous extrapolated point to construct the current extrapolation,
thereby requiring only one new operator evaluation per iteration:
\begin{equation}
	\label{eqa:SEG}
	\begin{aligned}
		\hat{x}_k &= x_k-\gamma_kF(\hat{x}_{k-1}),\\
		x_{k+1} &= x_k-\gamma_kF(\hat{x}_k).
	\end{aligned}
	\tag{SEG}
\end{equation}

The Polyak-type extragradient (PolyakEG) method \cite{yoon2026polyak, choudhury2025next} incorporates an adaptive
stepsize into the correction step of EG, with the stepsize determined by
the extrapolation and the corresponding operator evaluation:
\begin{equation}
	\label{eqa:PolyakEG}
	\begin{aligned}
		\hat{x}_k &= x_k-\gamma_kF(x_k),\\
		\alpha_k &=
		\frac{
			\left\langle F(\hat{x}_k),x_k-\hat{x}_k\right\rangle
		}{
			\left\|F(\hat{x}_k)\right\|^2
		},\\
		x_{k+1} &= x_k-\alpha_kF(\hat{x}_k).
	\end{aligned}
	\tag{PolyakEG}
\end{equation}

\subsection{Polyak step-size for single-call extragradient}
\label{subsec:polyak-stepsize}

We next derive the Polyak-type adaptive stepsize for SPolyakEG.
The derivation follows the same distance-to-solution principle as
PolyakEG, while taking into account the recycled operator evaluation.

For the nontrivial case $F(\hat{x}_k)\neq0$, consider the update
$
x_{k+1}=x_k-\alpha_kF(\hat{x}_k)$.
If $F(\hat{x}_k)=0$, then $\hat{x}_k$ is already a solution of
\eqref{eq:root}, and the algorithm terminates. Hence, we consider
the nontrivial case $F(\hat{x}_k)\neq0$ throughout the analysis.

%

Let $x_* \in X^*$ and consider the single-call extragradient update
$
\hat x_k  =x_k-\gamma_k F(\hat x_{k-1}), ~
x_{k+1}  =x_k-\alpha_kF(\hat x_k).
$
For any $\alpha_k\geq 0$, we have
\begin{align*}
	\|x_{k+1}-x^*\|^2
	={}&
	\|x_k-x^*\|^2
	-2\alpha_k
	\langle F(\hat x_k),x_k-\hat x_k\rangle
	\nonumber\\
	&-2\alpha_k
	\langle F(\hat x_k),\hat x_k-x^*\rangle
	+\alpha_k^2\|F(\hat x_k)\|^2.
\end{align*}
Since $F$ is monotone and $F(x^*)=0$, we know that $
\langle F(\hat x_k),\hat x_k-x^*\rangle\geq 0$.
Therefore,
\begin{align}
	\|x_{k+1}-x^*\|^2
	\leq 
	\|x_k-x^*\|^2
	-2\alpha_k
	\langle F(\hat x_k),x_k-\hat x_k\rangle  
	+\alpha_k^2\|F(\hat x_k)\|^2.
	\label{eq:single-call-upper-bound}
\end{align}

Following the Polyak-type principle, we minimize the upper bound
in \eqref{eq:single-call-upper-bound} with respect to
$\alpha_k\geq0$. This gives $
	\alpha_k
	=
	\frac{
		\left[
		\langle F(\hat x_k),x_k-\hat x_k\rangle
		\right]_+
	}{
		\|F(\hat x_k)\|^2
	}$. 
For single-call extragradient, 
$
x_k-\hat x_k
=
\gamma_kF(\hat x_{k-1})$,
and hence $
	\alpha_k
	=
	\gamma_k
	\frac{
		\left[
		\langle F(\hat x_k),F(\hat x_{k-1})\rangle
		\right]_+
	}{
		\|F(\hat x_k)\|^2
	}$. 

Under the delayed critical condition introduced in
Section~\ref{sec:convergence}, the inner product is positive whenever
$F(\hat{x}_k)\neq0$. Therefore, the projection onto the nonnegative
half-line can be omitted in the subsequent analysis.

\begin{equation*}
	\label{eq:single-call-polyak-stepsize}
		\alpha_k
		=
		\gamma_k
		\frac{
			\langle F(\hat x_k),F(\hat x_{k-1})\rangle
		}{
			\|F(\hat x_k)\|^2
		}.
\end{equation*}

The essential difference from PolyakEG is that the current operator
value $F(x_k)$ used in the extrapolation step is replaced by the
previously stored value $F(\hat x_{k-1})$. Consequently, the same
Polyak-type minimization principle leads to an adaptive stepsize
involving two consecutive extrapolation operator values.

\subsection{Single-call Polyak-tpype extragradient algorithm}

The proposed single-call Polyak-type extragradient method,
referred to as SPolyakEG, is summarized in
Algorithm~\ref{alg:SPolyakEG}.

\begin{algorithm}[H]
	\small
	\caption{\algname{\textcolor{PineGreen}{SPolyakEG}}}
	\label{alg:SPolyakEG}
	\begin{algorithmic}[1]
		\REQUIRE Initial point $x_0 \in \R^d$, initial extrapolation step size $\gamma_0 > 0$, and a rule for choosing extrapolation steps $\{\gamma_k\}_{k=1}^\infty$.
		
		\STATE \textbf{Step 0 (Two-evaluation starting step):}
		\STATE $\hx_0 = x_0 - \gamma_0 F(x_0)$.
		\vspace{2mm}
		\STATE \colorbox{green!20}{$\alpha_0 = \frac{\la F(\hx_0), x_0 - \hx_0 \ra}{\|F(\hx_0)\|^2}$.}
		\vspace{2mm}
		\STATE $x_1 = x_0 - \alpha_0 F(\hx_0)$.
		\STATE Store the operator value $F(\hx_0)$.
		\vspace{3mm}
		
		\FOR{$k = 1, 2,...,K$}
		\STATE $\hx_k = x_k - \gamma_k F(\hx_{k-1})$.
		\vspace{2mm}
		\STATE \colorbox{green!20}{$\alpha_k = \frac{\la F(\hx_k), x_k - \hx_k \ra}{\|F(\hx_k)\|^2}$.}
		\vspace{2mm}
		\STATE $x_{k+1} = x_k - \alpha_k F(\hx_k)$.
		\STATE Store the operator value $F(\hx_k)$.
		\ENDFOR
	\end{algorithmic}
\end{algorithm}

The main idea of SPolyakEG is to reuse the operator evaluation
$F(\hat{x}_{k-1})$ computed in the previous iteration when constructing
the current extrapolation point, thereby avoiding a new evaluation of
$F(x_k)$. Thus, after the initialization step, only one new operator
evaluation, namely $F(\hat{x}_k)$, is required at each iteration.
For notational convenience, let $g_k:=F(\hat{x}_k)$. 
Then, for $k\geq 1$, the SPolyakEG iteration can be written compactly as
\begin{equation}
	\begin{aligned}
		\hat x_k &= x_k-\gamma_k g_{k-1}, \\
		\alpha_k &=
		\frac{
			\langle g_k,x_k-\hat x_k\rangle
		}	{\|g_k\|^2}, \\
	x_{k+1} &= x_k-\alpha_k g_k .	
	\end{aligned} \tag{SPolyakEG}
\end{equation}
Since $g_{k-1}$ is stored from the previous iteration, only one new
operator evaluation, namely $g_k=F(\hat{x}_k)$, is required at each
iteration for $k\geq 1$. This compact formulation will be used
throughout the subsequent theoretical analysis.
 
\section{Convergence analysis}
\label{sec:convergence}

We establish the convergence properties of SPolyakEG under the
delayed critical condition introduced in the previous section.
We first derive several consequences of this condition, including
bounds on consecutive operator evaluations and the Polyak stepsize.
We then establish a sufficient condition under which the delayed
critical condition is guaranteed for Lipschitz continuous operators.
Based on these properties, we derive a Fej\'er-type descent inequality
and an $O(1/K)$ rate for the squared operator residual. Finally, under
$\mu$-strong monotonicity, we establish linear convergence of the
iterates.

\subsection{Delayed critical condition}
\label{subsec:delayed-critical-condition}

We begin by examining the consequences of the delayed critical
condition for the consecutive operator evaluations used by SPolyakEG.
In particular, this condition provides upper and lower bounds on
$\|g_k\|$ relative to $\|g_{k-1}\|$ and guarantees a positive inner
product between $g_k$ and $g_{k-1}$. These properties are essential for
establishing the positivity and bounds of the Polyak stepsize
$\alpha_k$.
We impose the following delayed critical condition.

\begin{assumption}[Delayed critical condition]\label{assumption:critical}	
	There exists a constant $A\in(0,1)$ such that	
	\begin{equation}
		\|g_k-g_{k-1}\|
		\leq
		A\|g_{k-1}\|,
		\qquad k \geq 1.
	\end{equation}	
\end{assumption}

%
%
%

We first establish a sufficient condition for the delayed critical
condition under the Lipschitz continuity assumption.

\begin{lemma}[Bounds of consecutive operator values]
	\label{lem:operator-bound}
	
Assume that the delayed critical condition holds, i.e.,
	\begin{equation*}
		\label{eq:delayed-condition}
		\|g_k-g_{k-1}\|
		\leq
		A\|g_{k-1}\|,
		\qquad
		0 < A < 1, ~ K \geq 1.
	\end{equation*}
	Then the consecutive operator values satisfy
	\begin{equation*}
		\label{eq:g-bound}
		\frac{1}{1+A}\|g_k\|
		\leq
		\|g_{k-1}\|
		\leq
		\frac{1}{1-A}\|g_k\|.
	\end{equation*}
	Moreover, the adaptive stepsize satisfies
	\begin{equation*}
		\label{eq:alpha-upper}
		\alpha_{k-1}
		\leq
		\frac{\gamma_{k-1}}{1-A}.
	\end{equation*}
	
\end{lemma}

\begin{proof}
	
	First, by the triangle inequality, we have
	\begin{equation*}
	\begin{aligned}
		\|g_k\|
		=
		\|g_{k-1}+(g_k-g_{k-1})\|
		\leq
		\|g_{k-1}\|
		+
		\|g_k-g_{k-1}\|.
	\end{aligned}
    \end{equation*}
	Using the delayed critical condition,	
     we obtain
	\begin{equation*}
	\|g_k\|
	\leq
	(1+A)\|g_{k-1}\|.
    \end{equation*}
	This yields,
	\begin{equation}
		\label{eq:lower-g}
		\|g_{k-1}\|
		\geq
		\frac{1}{1+A}\|g_k\|.
	\end{equation}
 
	On the other hand,
	\begin{equation*}
	\begin{aligned}
		\|g_{k-1}\|
		 =
		\|g_k-(g_k-g_{k-1})\|
        \leq
		\|g_k\|
		+
		\|g_k-g_{k-1}\|.
	\end{aligned}
	\end{equation*}
	Again, using
	\autoref{assumption:critical},
	\begin{equation*}
	\|g_{k-1}\|
	\leq
	\|g_k\|
	+
	A\|g_{k-1}\|.
	\end{equation*}
    This implies that,
	\begin{equation*}
	(1-A)\|g_{k-1}\|
	\leq
	\|g_k\|,
	\end{equation*}
	This yields
	\begin{equation}
		\label{eq:upper-g}
		\|g_{k-1}\|
		\leq
		\frac{1}{1-A}\|g_k\|.
	\end{equation}
	Combining
	\eqref{eq:lower-g}
	and
	\eqref{eq:upper-g},
	we obtain
	\begin{equation} \label{eq:lower-upper}
	\frac{1}{1+A}\|g_k\|
	\leq
	\|g_{k-1}\|
	\leq
	\frac{1}{1-A}\|g_k\|.
	\end{equation}

	Next, according to the definition of the Polyak stepsize,  we have
	\begin{equation*}
		\label{eq:alpha-definition}
		\alpha_{k-1}
		=
		\gamma_{k-1}
		\frac{
			\langle g_{k-1},g_{k-2}\rangle
		}
		{\|g_{k-1}\|^2}.
	\end{equation*}
	By the Cauchy--Schwarz inequality, we obtain
	\begin{equation*}
	\langle g_{k-1},g_{k-2}\rangle
	\leq
	\|g_{k-1}\|
	\|g_{k-2}\|.
	\end{equation*}
	It then follows that,
	\begin{equation*}
	\alpha_{k-1}
	\leq
	\gamma_{k-1}
	\frac{
		\|g_{k-2}\|
	}
	{\|g_{k-1}\|}.
	\end{equation*}
	Applying the upper bound in
	\eqref{eq:lower-upper} with index $k-1$ gives
	\begin{equation*}
	\|g_{k-2}\|
	\leq
	\frac{1}{1-A}
	\|g_{k-1}\|.
	\end{equation*}
	This implies that,
	\begin{equation*}
	\alpha_{k-1}
	\leq
	\gamma_{k-1}
	\frac{
		\frac{1}{1-A}
		\|g_{k-1}\|
	}
	{\|g_{k-1}\|}
	,
	\end{equation*}
	this yields
	\begin{equation*}
		\alpha_{k-1}
		\leq
		\frac{\gamma_{k-1}}{1-A}.
	\end{equation*}
	The proof is complete.
	
\end{proof}

\subsection{Sufficient conditions for the delayed critical condition}
\label{subsec:sufficient-delayed-condition}

We next establish a sufficient condition for the delayed critical
condition when $F$ is Lipschitz continuous. The proof relies on an
induction argument that explicitly accounts for the one-step delayed
operator evaluation in SPolyakEG.


\begin{theorem}[Sufficient condition for the delayed critical condition]
	\label{thm:delayed_condition}
	 Define $g_{-1}=F(x_0)$. Assume that $F$ is $L$-Lipschitz continuous with $L>0$ and let
	$0<A<1$. And the initial
	extrapolation stepsize satisfies $L\gamma_0\leq A$,
	and for $k\geq 1$, the extrapolation stepsizes satisfy
	\begin{equation*}
		\label{eq:gamma_condition}
		L\left(
		\gamma_k+
		\frac{2\gamma_{k-1}}{1-A}
		\right)
		\leq A.
	\end{equation*}
	Then the delayed critical condition
	\begin{equation*}
		\label{eq:delayed_condition}
		\|g_k-g_{k-1}\|
		\leq
		A\|g_{k-1}\|,
		\qquad k\geq0,
	\end{equation*}
	holds for the SPolyakEG iterates.
\end{theorem}

\begin{proof}
	We prove the delayed critical condition by induction.
	
	For $k=0$, since
	$
	g_{-1}=F(x_0),
	\hat{x}_0=x_0-\gamma_0 g_{-1},
	$
	the Lipschitz continuity of $F$ gives
	\begin{equation} \label{eq:sdcc1}
	\begin{aligned}
		\|g_0-g_{-1}\|
		 =
		\|F(\hat{x}_0)-F(x_0)\| 
		\leq
		L\|\hat{x}_0-x_0\| 
		=
		L\gamma_0\|g_{-1}\|
		\leq
		A\|g_{-1}\|,
	\end{aligned}
	\end{equation}
	where the last inequality in \eqref{eq:sdcc1} follows from $L\gamma_0\leq A$.
	Hence, the delayed critical condition holds for $k=0$.
	
	Now suppose that, for some $k\geq 1$, the delayed critical condition
	holds at iteration $k-1$, i.e.,
	\begin{equation}
		\label{eq:induction_hypothesis}
		\|g_{k-1}-g_{k-2}\|
		\leq
		A\|g_{k-2}\|.
	\end{equation}
	Then, by the reverse triangle inequality,
	\[
	\|g_{k-1}\|
	\geq
	\|g_{k-2}\|-\|g_{k-1}-g_{k-2}\|
	\geq
	(1-A)\|g_{k-2}\|,
	\]
	and this yields
	\begin{equation}
		\label{eq:g_kminus2_bound}
		\|g_{k-2}\|
		\leq
		\frac{1}{1-A}\|g_{k-1}\|.
	\end{equation}
By the induction hypothesis,
$
\|g_{k-1}-g_{k-2}\|
\leq A\|g_{k-2}\|$,
and hence
	\begin{equation*}
\|g_{k-1}\|
\leq
\|g_{k-2}\|
+
\|g_{k-1}-g_{k-2}\|
\leq
(1+A)\|g_{k-2}\|.
	\end{equation*}
This implies that,
	\begin{equation}  
		\label{eq:g_kminus2_lbound}
\|g_{k-2}\|^2
\geq
\frac{1}{(1+A)^2}\|g_{k-1}\|^2.
	\end{equation}

Moreover, by the definition of the adaptive stepsize,
$
\alpha_{k-1}
=
\gamma_{k-1}
\frac{\langle g_{k-1},g_{k-2}\rangle}
{\|g_{k-1}\|^2}.
$
The induction hypothesis \eqref{eq:induction_hypothesis} implies
	\begin{equation*} 
\begin{aligned}
	\langle g_{k-1},g_{k-2}\rangle
	&=
	\frac{1}{2}
	\left(
	\|g_{k-1}\|^2
	+
	\|g_{k-2}\|^2
	-
	\|g_{k-1}-g_{k-2}\|^2
	\right)\\
	&\geq
	\frac{1}{2}
	\left(
	\|g_{k-1}\|^2
	+
	(1-A^2)\|g_{k-2}\|^2
	\right)\\
	&\geq
	\frac{1}{2}
	\left(
	1+\frac{1-A}{1+A}
	\right)
	\|g_{k-1}\|^2\\
	&=
	\frac{1}{1+A}\|g_{k-1}\|^2.
\end{aligned}
	\end{equation*}
The second inequality follows from \eqref{eq:g_kminus2_lbound}.
	Consequently,
	\begin{equation}
		\label{eq:alpha_positive}
		\alpha_{k-1}
		\geq
		\frac{\gamma_{k-1}}{1+A}
		>0.
	\end{equation}
	Since $\gamma_k>0$, it follows that
	$
	\alpha_{k-1}+\gamma_k>0.
	$
	Thus,
	$
	|\alpha_{k-1}+\gamma_k|
	=
	\alpha_{k-1}+\gamma_k.
	$
	
	On the other hand, by the Cauchy--Schwarz inequality and
	\eqref{eq:g_kminus2_bound}, we obtain
	\begin{equation} \label{eq:alpha_upbound}
	\begin{aligned}
		\alpha_{k-1}
		&=
		\gamma_{k-1}
		\frac{\langle g_{k-1},g_{k-2}\rangle}
		{\|g_{k-1}\|^2}\\
		&\leq
		\gamma_{k-1}
		\frac{\|g_{k-2}\|}{\|g_{k-1}\|}\\
		&\leq
		\frac{\gamma_{k-1}}{1-A}.
	\end{aligned}
	\end{equation}
	Using
	$
	\hat{x}_k
	=
	x_k-\gamma_k g_{k-1},
	x_k
	=
	x_{k-1}-\alpha_{k-1}g_{k-1},
	\hat{x}_{k-1}
	=
	x_{k-1}-\gamma_{k-1}g_{k-2},
	$
	we obtain
	\begin{equation*}
	\hat{x}_k-\hat{x}_{k-1}
	=
	-(\alpha_{k-1}+\gamma_k)g_{k-1}
	+
	\gamma_{k-1}g_{k-2}.
	\end{equation*}
	This in turn gives,
	\begin{equation*}
	\begin{aligned}
		\|\hat{x}_k-\hat{x}_{k-1}\|
		&=
		\left\|
		-(\alpha_{k-1}+\gamma_k)g_{k-1}
		+
		\gamma_{k-1}g_{k-2}
		\right\|\\
		&\leq
		|\alpha_{k-1}+\gamma_k|\|g_{k-1}\|
		+
		\gamma_{k-1}\|g_{k-2}\|\\
		&=
		(\alpha_{k-1}+\gamma_k)\|g_{k-1}\|
		+
		\gamma_{k-1}\|g_{k-2}\|\\
		&\leq
		\left(
		\gamma_k
		+
		\frac{\gamma_{k-1}}{1-A}
		\right)\|g_{k-1}\|
		+
		\frac{\gamma_{k-1}}{1-A}\|g_{k-1}\|\\
		&=
		\left(
		\gamma_k
		+
		\frac{2\gamma_{k-1}}{1-A}
		\right)\|g_{k-1}\|.
	\end{aligned}
	\end{equation*}
 The second inequality follows from \eqref{eq:g_kminus2_bound} and \eqref{eq:alpha_upbound}.
	Finally, since $F$ is $L$-Lipschitz continuous,
	\[
	\begin{aligned}
		\|g_k-g_{k-1}\|
		&=
		\|F(\hat{x}_k)-F(\hat{x}_{k-1})\|\\
		&\leq
		L\|\hat{x}_k-\hat{x}_{k-1}\|\\
		&\leq
		L\left(
		\gamma_k
		+
		\frac{2\gamma_{k-1}}{1-A}
		\right)
		\|g_{k-1}\|\\
		&\leq
		A\|g_{k-1}\|.
	\end{aligned}
	\]
	Therefore, the delayed critical condition holds at iteration $k$.
	By induction, it holds for all $k\geq 0$.
\end{proof}

For the constant stepsize case $\gamma_k=\gamma$, the above
condition reduces to
\begin{equation}
	\label{eq:constant_gamma}
		\gamma
		\leq
		\frac{A(1-A)}
		{L(3-A)}.
\end{equation}

\begin{remark}
	\label{rem:constant_stepsize}
	The stepsize bound in \eqref{eq:constant_gamma} is a sufficient
	condition induced by the above proof. In particular, the bound may be
	conservative due to the use of the triangle inequality and the
	intermediate estimates in the inductive argument. We do not claim
	that it is necessary or tight.
\end{remark}

\subsection*{Comparison with PolyakEG}

For the original PolyakEG method, the critical condition is
\begin{equation*}
	\label{eq:original_condition}
	\|F(\hat{x}_k)-F(x_k)\|
	\leq
	A\|F(x_k)\|.
\end{equation*}
Since
$
\hat{x}_k=x_k-\gamma_kF(x_k),
$
the Lipschitz continuity of $F$ gives
\begin{equation*}
	\|F(\hat{x}_k)-F(x_k)\|
	\leq
	L\gamma_k\|F(x_k)\|.
\end{equation*}
Therefore, the condition
\begin{equation*}
	\label{eq:peg_stepsize}
	\gamma_k
	\leq
	\frac{A}{L} 
\end{equation*}
is sufficient to guarantee the critical condition of PolyakEG.

In contrast, for SPolyakEG, our inductive analysis yields the
sufficient condition
\begin{equation*}
	\label{eq:spolyakeg_general_stepsize}
	L\left(
	\gamma_k+
	\frac{2\gamma_{k-1}}{1-A}
	\right)
	\leq A.
\end{equation*}
In the constant stepsize case $\gamma_k=\gamma$, this reduces to
\begin{equation}
	\label{eq:spolyakeg_stepsize}
	\gamma
	\leq
	\frac{A(1-A)}
	{L(3-A)}.
\end{equation}

Thus, if we compare the resulting sufficient stepsize upper bounds,
\[
\gamma_{\rm PEG}^{\max}
=
\frac{A}{L},
\qquad
\gamma_{\rm SP}^{\max}
=
\frac{A(1-A)}{L(3-A)},
\]
their ratio is
\begin{equation}
	\label{eq:stepsize_ratio}
	\frac{\gamma_{\rm SP}^{\max}}
	{\gamma_{\rm PEG}^{\max}}
	=
	\frac{1-A}{3-A}
	<1,
	\qquad A\in(0,1).
\end{equation}

Hence, under the sufficient conditions established by our analysis,
the stepsize bound for SPolyakEG is more restrictive than the
corresponding bound obtained for PolyakEG. This difference reflects
the additional technical restriction introduced by the delayed
operator evaluation in our inductive argument. Since the bound in
\eqref{eq:spolyakeg_stepsize} is only a sufficient condition and may
be conservative, the ratio in \eqref{eq:stepsize_ratio} should not
be interpreted as the ratio of the maximal stable stepsizes of the
two methods.

\subsection{Fej\'er convergence and sublinear rate}
\label{subsec:fejer-convergence}

Under the delayed critical condition, the Polyak stepsize is positive
and admits the bounds established above. We now use these properties
to derive a Fej\'er-type descent inequality for the iterates. This
inequality yields a sublinear convergence rate for the squared
operator residual evaluated at the extrapolation points.

We first establish a lower bound for the Polyak stepsize, which will be used in deriving the subsequent descent inequality.

\begin{lemma}\label{lemma:inner}
	If the delayed critical condition holds,  then
	\begin{equation*}
		\langle g_k,g_{k-1}\rangle
		\geq
		\frac{1}{1+A}\|g_k\|^2 .
	\end{equation*}
\end{lemma}

\begin{proof}
	Using the identity
	$	2\langle a,b\rangle
		=
		\|a\|^2+\|b\|^2-\|a-b\|^2 $,
	with $a=g_k$ and $b=g_{k-1}$, we have
	\begin{equation*}
		\begin{aligned}
			2\langle g_k,g_{k-1}\rangle =
			\|g_k\|^2+\|g_{k-1}\|^2-\|g_k-g_{k-1}\|^2 .
		\end{aligned}
	\end{equation*}
	By Assumption \ref{assumption:critical}, we have
	\begin{equation*}
		\|g_k-g_{k-1}\|^2
		\leq
		A^2\|g_{k-1}\|^2 .
	\end{equation*}
     It then follows that,
	\begin{equation} \label{eq:gkgk-1}
		2\langle g_k,g_{k-1}\rangle
		\geq
		\|g_k\|^2
		+
		(1-A^2)\|g_{k-1}\|^2 .
	\end{equation}

   In addition,
	\begin{equation*}
		\|g_k\|
		\leq
		\|g_k-g_{k-1}\|
		+
		\|g_{k-1}\|
		\leq
		(1+A)\|g_{k-1}\|,
	\end{equation*}
	this implies that,
	\begin{equation} \label{eq:alphak_lowerb}
		\|g_{k-1}\|
		\geq
		\frac1{1+A}\|g_k\|.
	\end{equation}
	Substituting \eqref{eq:alphak_lowerb} into the previous inequality \eqref{eq:gkgk-1} gives
	\begin{equation*}
		\begin{aligned}
			2\langle g_k,g_{k-1}\rangle
			&\geq
			\|g_k\|^2
			+
			(1-A^2)
			\frac1{(1+A)^2}
			\|g_k\|^2
			\\
			&=
			\frac{2}{1+A}
			\|g_k\|^2 .
		\end{aligned}
	\end{equation*}
    It then follows that
	\begin{equation*}
		\langle g_k,g_{k-1}\rangle
		\geq
		\frac1{1+A}\|g_k\|^2 .
	\end{equation*}
	
\end{proof}

\begin{lemma} [Lower bound of Polyak stepsize] 
	\label{lemma:alpha-lower-bound}	
	IF the delayed critical condition holds, then the Polyak stepsize generated by SPolyakEG satisfies
	\begin{equation*}
		\label{eq:alpha}
		\alpha_k
		\geq
		\frac{\gamma_k}{1+A} >0.
	\end{equation*}	
\end{lemma}

\begin{proof}
	Since $ x_k-\hat x_k=\gamma_k g_{k-1}$,	form the definition of $\alpha_k$,
	we have
	\begin{equation*}
		\alpha_k
		=
		\gamma_k
		\frac{
			\langle g_k,g_{k-1}\rangle
		}
		{\|g_k\|^2}.
	\end{equation*}	 
	From Lemma~\ref{lemma:inner}, we arrive at
	\begin{equation*}
		\alpha_k
		\geq
		\frac{\gamma_k}{1+A}.
	\end{equation*}	 
    Since $\gamma_k>0$, we conclude that $
	\alpha_k>0$.
\end{proof}

Using the lower bound of the Polyak stepsize, we next establish the
following descent inequality.

\begin{theorem}[Fej\'er-type descent]
	\label{theorem:fejer}
	
	Suppose that $F$ satisfies the monotonicity condition 
	and that $x_\star$ satisfies $		F(x_\star)=0$.
Then the iterates generated by SPolyakEG satisfy
	\begin{equation*}
		\label{eq:descent}
		\begin{aligned}
			\|x_{k+1}-x_\star\|^2
			\leq 
			\|x_k-x_\star\|^2
			-\alpha_k\gamma_k
			\langle g_k,g_{k-1}\rangle .
		\end{aligned}
	\end{equation*}
	
\end{theorem}

\begin{proof}
	
	By the update rule $
		x_{k+1}=x_k-\alpha_k g_k $, 
	we obtain
	\begin{equation} \label{eq:fej1}
		\begin{aligned}
			\|x_{k+1}-x_\star\|^2
			=
			\|x_k-x_\star\|^2
			-2\alpha_k
			\langle g_k,x_k-x_\star\rangle
			+\alpha_k^2\|g_k\|^2 .
		\end{aligned}
	\end{equation}
	Since  $
		x_k-x_\star
		=
		(x_k-\hat x_k)
		+
		(\hat x_k-x_\star)$, 
	we have 
	\begin{equation*}
		\begin{aligned}
			\langle g_k,x_k-x_\star\rangle
			=
			\langle g_k,x_k-\hat x_k\rangle
			+
			\langle g_k,\hat x_k-x_\star\rangle .
		\end{aligned}
	\end{equation*}
	Because 
		$\hat x_k=x_k-\gamma_k g_{k-1}$, 
	it follows that
	$	x_k-\hat x_k
		=
		\gamma_k g_{k-1}$.
	It then follows that,
	\begin{equation} \label{eq:fej2}
		\begin{aligned}
			\langle g_k,x_k-x_\star\rangle
			=
			\gamma_k
			\langle g_k,g_{k-1}\rangle
			+
			\langle g_k,\hat x_k-x_\star\rangle .
		\end{aligned}
	\end{equation}
	
Now we use the root-finding property.
	Since $
		g_k=F(\hat x_k)$  
	and
	$F(x_\star)=0$, 
	we have
	\begin{equation*} 
		\begin{aligned}
			\langle g_k,\hat x_k-x_\star\rangle
			=&
			\langle
			F(\hat x_k)-F(x_\star),
			\hat x_k-x_\star
			\rangle .
		\end{aligned}
	\end{equation*}
	By the monotonicity of $F$, we obtain
	\begin{equation*} 
		\langle g_k,\hat x_k-x_\star\rangle
		\geq0 .
	\end{equation*}
	From the above relation and \eqref{eq:fej2}, we arrive at 
	\begin{equation} \label{eq:fej3}
		\langle g_k,x_k-x_\star\rangle
		\geq
		\gamma_k
		\langle g_k,g_{k-1}\rangle .
	\end{equation}
	Substituting \eqref{eq:fej3} into the distance recursion \eqref{eq:fej1}, we obtain
	\begin{equation}
		\begin{aligned} \label{eq:fej4}
			\|x_{k+1}-x_\star\|^2
			\leq
			\|x_k-x_\star\|^2
			-2\alpha_k\gamma_k
			\langle g_k,g_{k-1}\rangle
			+\alpha_k^2\|g_k\|^2 .
		\end{aligned}
	\end{equation}

	By the definition of $\alpha_k$, we have
	\begin{equation*}
		\alpha_k\|g_k\|^2
		=
		\gamma_k
		\langle g_k,g_{k-1}\rangle ,
	\end{equation*}
    this yields,
	\begin{equation*}
		\alpha_k^2\|g_k\|^2
		=
		\alpha_k\gamma_k
		\langle g_k,g_{k-1}\rangle .
	\end{equation*}
    From the above relation and \eqref{eq:fej4}, we consequently obtain
	\begin{equation*}
			\|x_{k+1}-x_\star\|^2
			\leq
			\|x_k-x_\star\|^2
			-
			\alpha_k\gamma_k
			\langle g_k,g_{k-1}\rangle .
	\end{equation*}
	
\end{proof}

%
%
%
%
%
%
%
%
%
%
%
%
%
%
%

\begin{lemma}[Lower bound of the Polyak descent term]
	\label{lem:polyak-descent-term}
	
	 If the adaptive stepsize of SPolyakEG method is defined as
	\begin{equation*}
		\label{eq:alpha-definition}
		\alpha_k
		=
		\gamma_k
		\frac{
			\langle g_k,g_{k-1}\rangle
		}
		{\|g_k\|^2}.
	\end{equation*}
	Then under the delayed critical condition,
	we have
	\begin{equation*}
		\label{eq:polyak-term-bound}
		\alpha_k\gamma_k
		\langle g_k,g_{k-1}\rangle
		\geq
		\frac{\gamma_k^2}{(1+A)^2}
		\|g_k\|^2 .
	\end{equation*}
	
\end{lemma}

\begin{proof}
	
	Substituting the definition of $\alpha_k$,
	we obtain
	\begin{equation} \label{eq:lbp1}
	\begin{aligned}
		\alpha_k\gamma_k
		\langle g_k,g_{k-1}\rangle
		&=
		\gamma_k
		\frac{
			\langle g_k,g_{k-1}\rangle
		}
		{\|g_k\|^2}
		\gamma_k
		\langle g_k,g_{k-1}\rangle
		\\
		&=
		\gamma_k^2
		\frac{
			\langle g_k,g_{k-1}\rangle^2
		}
		{\|g_k\|^2}.
	\end{aligned}
	\end{equation}
	By Lemma~\ref{lem:operator-bound}, it follows that
	\begin{equation*}
	\langle g_k,g_{k-1}\rangle
	\geq
	\frac1{1+A}\|g_k\|^2 .
	\end{equation*}
	This implies that,
	\begin{equation*} 
	\langle g_k,g_{k-1}\rangle^2
	\geq
	\frac1{(1+A)^2}
	\|g_k\|^4 .
	\end{equation*}
    Combining the above inequality with \eqref{eq:lbp1}, we arrive at,
	\begin{equation*}
	\begin{aligned}
		\alpha_k\gamma_k
		\langle g_k,g_{k-1}\rangle
		&=
		\gamma_k^2
		\frac{
			\langle g_k,g_{k-1}\rangle^2
		}
		{\|g_k\|^2}
		\\
		&\geq
		\gamma_k^2
		\frac{
			\frac1{(1+A)^2}
			\|g_k\|^4
		}
		{\|g_k\|^2}
		\\
		&=
		\frac{\gamma_k^2}{(1+A)^2}
		\|g_k\|^2 .
	\end{aligned}
	\end{equation*}
	This proves the desired inequality.
	
\end{proof}

By summing the descent inequality over the iterations, we then
establish the following $O(1/k)$ convergence rate for the squared
operator residual.

\begin{theorem}
	\label{thm:SPolyakEG_rate}
	
	Suppose that $F$ is monotone and
	Assumption~\ref{assumption:critical} holds.
	Let $\gamma_k\geq\underline{\gamma}>0$.
	Then the iterates generated by SPolyakEG satisfy
	
	\begin{equation*}
		\label{eq:SP_rate}
		\min_{0\leq k\leq K}
		\|F(\hat{x}_k)\|^2
		\leq
		\frac{
			(1+A)^2
			\|x_0-x_\star\|^2
		}
		{
			\underline{\gamma}^{\,2}(K+1)
		}.
	\end{equation*}

\end{theorem}

\begin{proof}
	
	From Fej\'er-type descent theorem \ref{theorem:fejer},  we arrive at
	\begin{equation} \label{eq:src1}  
			\|x_{k+1}-x_\star\|^2 \leq \|x_k-x_\star\|^2-
	\alpha_k\gamma_k
	\langle g_k,g_{k-1}\rangle .
	\end{equation}
	By Lemma~\ref{lemma:inner}, we have
	\[
	\langle g_k,g_{k-1}\rangle
	\geq
	\frac{1}{1+A}\|g_k\|^2 .
	\]
	It is worth noting that,
	\[
	\alpha_k
	=
	\gamma_k
	\frac{
		\langle g_k,g_{k-1}\rangle
	}
	{\|g_k\|^2}.
	\]
	From Lemma \ref{lem:polyak-descent-term}, we obtain
	\[
	\alpha_k\gamma_k
	\langle g_k,g_{k-1}\rangle
	\geq
	\frac{\gamma_k^2}{(1+A)^2}
	\|g_k\|^2 .
	\]
    Combining the above with \eqref{eq:src1}, we arrive at
	\[
	\frac{\gamma_k^2}{(1+A)^2}
	\|g_k\|^2
	\leq
	\|x_k-x_\star\|^2
	-
	\|x_{k+1}-x_\star\|^2 .
	\]
	Summing from $0$ to $K$ yields
	\[
	\sum_{k=0}^{K}
	\frac{\gamma_k^2}{(1+A)^2}
	\|g_k\|^2
	\leq
	\|x_0-x_\star\|^2 .
	\]
	Since $\gamma_k\geq\underline{\gamma}$, it can be seen that
	\[
	\frac{(K+1)\underline{\gamma}^2}
	{(1+A)^2}
	\min_{0\leq k\leq K}
	\|g_k\|^2
	\leq
	\|x_0-x_\star\|^2 .
	\]
	The result follows from $g_k=F(\hat{x}_k)$.
\end{proof}

\subsection{Linear convergence under $\mu$-strong monotonicity}
\label{subsec:linear-convergence} 

We finally consider the case in which $F$ is $\mu$-strongly monotone.
In this setting, the monotonicity term in the distance recursion
provides an additional contraction mechanism. By combining strong
monotonicity with the delayed critical condition and the stepsize
bounds established above, we obtain a linear convergence result for
the iterates.

\begin{theorem}[Linear convergence of SPolyakEG]
	\label{thm:SPolyakEG-linear}
	IF delayed critical Ccondition and $\mu$-Strong monotonicity 
	hold.
	And the extrapolation parameters satisfy
	\begin{equation*}
		\label{eq:gamma-condition-linear}
		\gamma_k\mu
		\leq
		\frac{A(1-A)}{2},
	\end{equation*}
	then the iterates generated by SPolyakEG satisfy
	\begin{equation*}
		\label{eq:linear-variable}
		\|x_{k+1}-x^\star\|^2
		\leq
		\left(
		1-
		\frac{2\gamma_k\mu}
		{(1+A)^2}
		\right)
		\|x_k-x^\star\|^2 .
	\end{equation*}
\end{theorem}

\begin{proof}
	Let	 $
	g_k=F(\hat{x}_k)$,
	from the update rule $
	x_{k+1}=x_k-\alpha_k g_k $, 
	we have
	\begin{equation}	\label{eq:basic-expansion}
		\begin{aligned}
			\|x_{k+1}-x^\star\|^2
			=
			\|x_k-x^\star\|^2
			-2\alpha_k
			\langle g_k,x_k-x^\star\rangle
			+
			\alpha_k^2\|g_k\|^2 .
		\end{aligned}
	\end{equation}
	Since $
	x_k-\hat{x}_k
	=
	\gamma_k g_{k-1}$,  it then follows that
	\begin{equation*}
	x_k-x^\star
	=
	\gamma_k g_{k-1}
	+
	(\hat{x}_k-x^\star).
    \end{equation*}
	This implies that
	\begin{equation}  \label{eq:inner-decomposition}
		\begin{aligned}
			\langle g_k,x_k-x^\star\rangle
			= 
			\gamma_k
			\langle g_k,g_{k-1}\rangle
			+
			\langle g_k,\hat{x}_k-x^\star\rangle .
		\end{aligned}
	\end{equation}
	Substituting \eqref{eq:inner-decomposition}
	into \eqref{eq:basic-expansion}, we obtain
	\begin{equation} \label{eq:lcr1}
		\begin{aligned}
			\|x_{k+1}-x^\star\|^2
			\leq
			\|x_k-x^\star\|^2
			-2\alpha_k\gamma_k
			\langle g_k,g_{k-1}\rangle
			-2\alpha_k
			\langle g_k,\hat{x}_k-x^\star\rangle
			+
			\alpha_k^2\|g_k\|^2 .
		\end{aligned}
	\end{equation}
	
	According to the definition of $\alpha_k$,
	\begin{equation} \label{eq:lcr2}
		\alpha_k^2\|g_k\|^2
		=
		\alpha_k\gamma_k
		\langle g_k,g_{k-1}\rangle .
	\end{equation}
	Upon substituting \eqref{eq:lcr2} into \eqref{eq:lcr1}, we arrive at 
	\begin{equation*}
		\begin{aligned}
			\|x_{k+1}-x^\star\|^2
			\leq 
			\|x_k-x^\star\|^2
			-\alpha_k\gamma_k
			\langle g_k,g_{k-1}\rangle
			-2\alpha_k
			\langle g_k,\hat{x}_k-x^\star\rangle .
		\end{aligned}
		\label{eq:descent2}
	\end{equation*}
	By $\mu$-strong monotonicity of $F$,  we have
	\[
	\langle g_k,\hat{x}_k-x^\star\rangle
	\geq
	\mu
	\|\hat{x}_k-x^\star\|^2,
	\]
	this implies that
	\begin{equation}
		\begin{aligned}
			\|x_{k+1}-x^\star\|^2
			\leq
			\|x_k-x^\star\|^2
			-\alpha_k\gamma_k
			\langle g_k,g_{k-1}\rangle
			-2\alpha_k\mu
			\|\hat{x}_k-x^\star\|^2 .
		\end{aligned}
		\label{eq:descent3}
	\end{equation}

	It is also worth noting that
	\[
	x_k-x^\star
	=
	(\hat{x}_k-x^\star)
	+\gamma_k g_{k-1},
	\]
	from Young's inequality,  we conclude that for any $\omega>0$,
	\begin{equation*}
		\|x_k-x^\star\|^2
		\leq
		(1+\omega)
		\|\hat{x}_k-x^\star\|^2
		+
		(1+\frac1\omega)
		\gamma_k^2
		\|g_{k-1}\|^2 .
	\end{equation*}
	We thus get
	\begin{equation} \label{eq:lcr3}
		\|\hat{x}_k-x^\star\|^2
		\geq
		\frac1{1+\omega}
		\|x_k-x^\star\|^2
		-
		\frac1\omega
		\gamma_k^2\|g_{k-1}\|^2 .
	\end{equation}
	Substituting \eqref{eq:lcr3} into
	\eqref{eq:descent3}  yields
	\begin{equation*} 
		\begin{aligned}
			\|x_{k+1}-x^\star\|^2
			\leq
			\left(
			1-\frac{2\alpha_k\mu}{1+\omega}
			\right)
			\|x_k-x^\star\|^2
			+
			\frac{2\alpha_k\mu}{\omega}
			\gamma_k^2
			\|g_{k-1}\|^2
			-
			\alpha_k\gamma_k
			\langle g_k,g_{k-1}\rangle .
		\end{aligned}
		\label{eq:descent4}
	\end{equation*}
   By the Cauchy--Schwarz inequality and the delayed critical condition,
we have
\begin{align*}
	\langle g_k,g_{k-1}\rangle
	&=
	\|g_{k-1}\|^2
	+
	\langle g_k-g_{k-1},g_{k-1}\rangle
	\nonumber\\
	&\geq
	\|g_{k-1}\|^2
	-
	\|g_k-g_{k-1}\|\|g_{k-1}\|
	\nonumber\\
	&\geq
	(1-A)\|g_{k-1}\|^2.
	\label{eq:inner_product_lower_bound_previous}
\end{align*}
%
%
%
   This implies that 
	\begin{equation} \label{eq:descent5}
	~~~~~~	\begin{aligned}
			\|x_{k+1}-x^\star\|^2
			\leq
			\left(
			1-\frac{2\alpha_k\mu}{1+\omega}
			\right)
			\|x_k-x^\star\|^2
			+
			\alpha_k\gamma_k
			\left(
			\frac{2\mu\gamma_k}{\omega}
			-(1-A)
			\right)
			\|g_{k-1}\|^2 .
		\end{aligned}		
	\end{equation}

	Choose $
	\omega=A $. 
	Under condition $
	\gamma_k\mu
	\leq
	\frac{A(1-A)}2 $, 
	the second term in \eqref{eq:descent5} is non-positive.
	This in turn gives
	\begin{equation} \label{eq:lcr4}
	\|x_{k+1}-x^\star\|^2
	\leq
	\left(
	1-\frac{2\alpha_k\mu}{1+A}
	\right)
	\|x_k-x^\star\|^2 .
	\end{equation}
	
	Invoking Lemma~\ref{lemma:alpha-lower-bound}, we obtain
	\begin{equation} \label{eq:lcr5}
	\alpha_k
	\geq
	\frac{\gamma_k}{1+A},
	\end{equation}
	plugging \eqref{eq:lcr5} into \eqref{eq:lcr4}, we get
	\begin{equation*}
	\|x_{k+1}-x^\star\|^2
	\leq
	\left(
	1-
	\frac{2\gamma_k\mu}{(1+A)^2}
	\right)
	\|x_k-x^\star\|^2 .
	\end{equation*}
Hence the sequence converges linearly.	 
%
%
%
%
%
\end{proof}

\begin{corollary}
	\label{thm:linear_rate}
	IF delayed critical Ccondition and $\mu$-Strong monotonicity 
	hold. 
	And the extrapolation parameters  $\gamma_k\equiv\gamma$  satisfy	
	\begin{equation*}
		0 < \gamma \leq \frac{A(1-A)}{2\mu},
	\end{equation*}
	then the iterates generated by SPolyakEG satisfy
	\begin{equation*}
		\label{eq:linear-rate}
		\|x_k-x^\star\|^2
		\leq
		\rho^k
		\|x_0-x^\star\|^2 ,
	\end{equation*}	
	where
$		\label{eq:rho}
		\rho
		=
		1-
		\frac{2\gamma\mu}{(1+A)^2}
		<1 $.
%
	
\end{corollary}

\begin{proof}
	For constant stepsize $\gamma_k=\gamma$, from Theorem \ref{thm:SPolyakEG-linear}, we have
	\begin{equation*}
    \|x_k-x^\star\|^2
    \leq
    \rho^k
    \|x_0-x^\star\|^2 ,
	\end{equation*}
where
   $ \rho=
    1-\frac{2\gamma\mu}{(1+A)^2}$.
\end{proof}


To highlight the similarities and differences between PolyakEG and
the proposed SPolyakEG method, we compare the two methods from several
perspectives, including the extrapolation direction, critical
condition, sufficient stepsize condition, convergence rate, and
oracle complexity. The main difference lies in the operator evaluation
used to construct the extrapolation point. Specifically, PolyakEG
requires a fresh evaluation of $F(x_k)$ at each iteration, whereas
SPolyakEG reuses the previously computed value $F(\hat{x}_{k-1})$.
This reuse reduces the number of new oracle evaluations per iteration
from two to one after initialization.

The convergence results for both methods are established under their
respective critical conditions. In particular, the stepsize bounds
reported in Table~\ref{tab:comparison} are sufficient conditions
obtained from the corresponding analyses and should not be interpreted
as necessary or tight stability bounds. For SPolyakEG, the more
restrictive stepsize condition results from the delayed operator
evaluation and the inductive argument used to establish the delayed
critical condition. The resulting comparison is summarized in
Table~\ref{tab:comparison}.

\begin{table}[H]
	\centering
	\small
	\setlength{\tabcolsep}{3pt} 
	\renewcommand{\arraystretch}{1.6} 
	\caption{Comparison between PolyakEG and SPolyakEG.}
	\label{tab:comparison}
	\begin{tabular}{ccc}
		\toprule
		& PolyakEG & SPolyakEG \\
		\midrule
		Extrapolation direction
		&
		$F(x_k)$
		&
		$F(\hat{x}_{k-1})$
		\\
		Critical condition
		&
		$\|F(\hat{x}_k)-F(x_k)\|
		\leq A\|F(x_k)\|$
		&
		$\|F(\hat{x}_k)-F(\hat{x}_{k-1})\|
		\leq A\|F(\hat{x}_{k-1})\|$
		\\
		Sufficient stepsize condition
		&
		$\displaystyle
		\gamma_k\leq\frac{A}{L}$
		&
		$\displaystyle
		L\left(
		\gamma_k+
		\frac{2\gamma_{k-1}}{1-A}
		\right)
		\leq A$
		\\
		Constant stepsize bound
		&
		$\displaystyle
		\gamma\leq\frac{A}{L}$
		&
		$\displaystyle
		\gamma\leq
		\frac{A(1-A)}{L(3-A)}$
		\\
		Residual rate
		&
		$\displaystyle
		\min_{0\leq k\leq K}
		\|F(\hat{x}_k)\|^2=O(1/K)$
		&
		$\displaystyle
		\min_{0\leq k\leq K}
		\|F(\hat{x}_k)\|^2=O(1/K)$
		\\
     	Strong monotonicity
		&
		Linear convergence
		&
		Linear convergence
		\\ 
		New oracle evaluations/iteration
		&
		2
		&
		1
		\\ 
		Oracle complexity
		&
		$O(2/\epsilon)$
		&
		$O(1/\epsilon)$
		\\
		\bottomrule
	\end{tabular}
\end{table}

\section{Numerical Experiments}
We now present numerical experiments to evaluate the performance of the proposed
\algname{SPolyakEG} algorithm.
We consider three standard minimax problems: a bilinear game, a convex-concave
minimax problem, and a Dirac-GAN problem.
For each problem, we compare \algname{SPolyakEG} with
\algname{EG}, \algname{SEG}, and \algname{PolyakEG} under the same
extrapolation step-size rule $\gamma_k<1/L$ and identical initialization.
Performance is evaluated in terms of the relative distance to the saddle point
and the average runtime.
The results are presented for each problem separately. 
 
\subsection{Bilinear game}

\begin{minipage}[t]{.50\textwidth}
	\vspace{0pt}
	We consider the 
	bilinear game \cite{keke2024training}
	\begin{align}
		\label{eqn:bilinear-plot}
		\underset{\theta \in \reals}{\min} \,\, \underset{\phi \in \reals}{\max} \,\, f(\theta ,\phi) =\theta \phi,
	\end{align} which is equivalent to \eqref{eq:root} with operator
$F(\theta,\phi)=(\phi,-\theta)$ and admits the unique saddle point
$(0,0)$.
\Cref{fig:bilinear3} depicts the trajectories over the first 50 iterations
produced by \algname{EG}, \algname{SEG}, \algname{PolyakEG}, and
\algname{SPolyakEG}, all initialized at
$(\theta_0,\phi_0)=(1,1)$.
All methods use the extrapolation step size
$\gamma_k=0.5<1/L$, where $L$ denotes the Lipschitz constant of $F$.
The trajectories show that the Polyak-type methods exhibit faster
per-iteration progress than their fixed-step counterparts.
In particular, \algname{SPolyakEG} converges faster than \algname{SEG}
in terms of iteration count, illustrating the role of the adaptive
Polyak correction in the proposed method.
\end{minipage}
\hfill
\begin{minipage}[t]{.48\textwidth}
	\vspace{0pt}
	\centering
	\includegraphics[width=\linewidth]{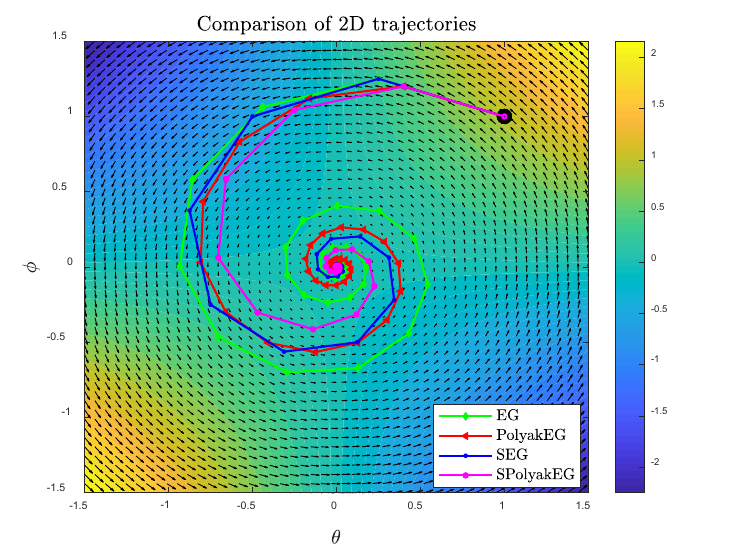}
	\captionsetup{font=small}
	\captionof{figure}{Trajectories of \algname{EG}, \algname{SEG},
		\algname{PolyakEG}, and \algname{SPolyakEG} for
		\eqref{eqn:bilinear-plot} with $\gamma_k=0.5<1/L$.
		The unique solution $(0,0)$ is marked in red. }
	\label{fig:bilinear3}
\end{minipage}

\begin{figure}[ht]
	\centering
	\begin{subfigure}[t]{.49\textwidth}
		\centering
		\includegraphics[width=\linewidth]{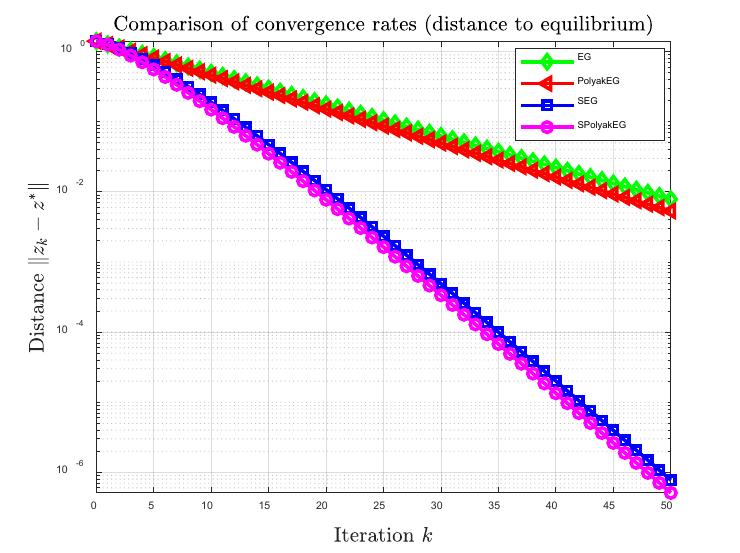}
		\caption{}
		\label{fig:bilinear2}
	\end{subfigure}
	\hfill
	\begin{subfigure}[t]{.49\textwidth}
		\centering
		\includegraphics[width=\linewidth]{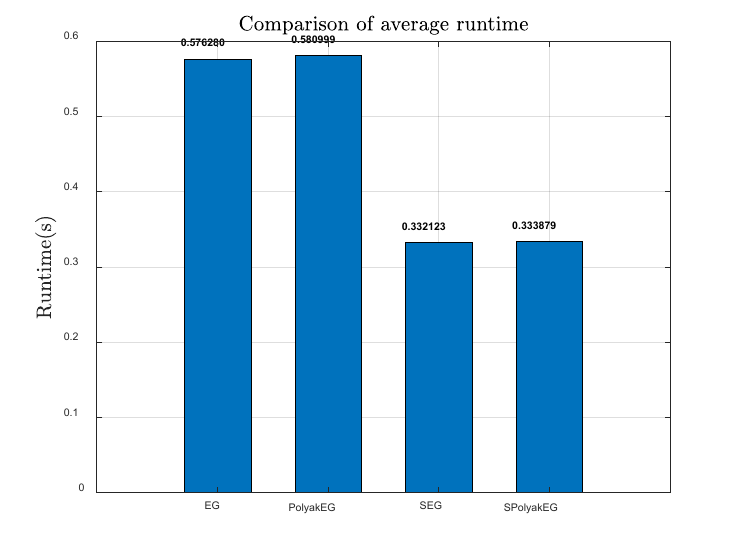}
		\caption{}
		\label{fig:bilinear1}
	\end{subfigure}
	\caption{Numerical results on the bilinear minimax game.
		Panel (a) compares the convergence behavior of the four methods
		in terms of the relative distance to the saddle point.
		Panel (b) compares their corresponding runtime.
	}
	\label{fig:bilinear1-plot}
\end{figure}

\Cref{fig:bilinear1-plot} provides a quantitative comparison of the four
algorithms on the bilinear game.
As shown in \Cref{fig:bilinear2}, \algname{SPolyakEG} exhibits the fastest
decay of the distance to the saddle point among the four tested methods,
followed by \algname{SEG}, \algname{PolyakEG}, and \algname{EG}.
The results also show that the Polyak-type methods exhibit faster
iteration-wise convergence than their fixed-step counterparts.
In particular, \algname{SPolyakEG} converges faster than \algname{SEG},
although both methods require only one new operator evaluation per iteration.
This comparison isolates the effect of the adaptive Polyak correction within
the single-call framework.

In terms of computational efficiency, the runtime results in
\Cref{fig:bilinear1} show that the single-call methods,
\algname{SEG} and \algname{SPolyakEG}, require less computational time
than the two-call methods, \algname{EG} and \algname{PolyakEG}.
In this experiment, the reduction in runtime is approximately $40\%$.
These results illustrate the computational benefit of reusing the
previously computed operator evaluation while retaining the adaptive
correction step.

\subsection{convex-concave min-max problem}

\begin{minipage}[t]{.50\textwidth}
	\vspace{0pt}
	We consider the 
	convex-concave min-max problem
	\begin{align}
		\label{eqn:cc-plot}
		\underset{\theta \in \reals}{\min} \,\, \underset{\phi \in \reals}{\max} \,\, f(\theta,\phi) = \frac{1}{2} \theta^2 + \theta \phi - \frac{1}{2} \phi^2, 
	\end{align}  which is equivalent to \eqref{eq:root} with $F(\theta,\phi) = (\theta+\phi, \phi-\theta)$ and admits the unique saddle point $(0,0)$.
    We plot in \Cref{fig:cc2-3} the first 20 iterations of \algname{EG}, \algname{SEG}, \algname{PolyakEG} and \algname{SPolyakEG} starting from $(\theta_0, \phi_0) = (1, 1)$.
   All methods use the extrapolation step-size $\gamma_k = 0.3 < 1/L$, where $L$ is the tight Lipschitz constant of $F$.
   We observe that \algname{SPolyakEG} achieves faster per-iteration convergence than \algname{SEG} with fixed $\alpha_k = \gamma_k$. This improvement highlights the role of the adaptive Polyak correction in the proposed single-call scheme.
\end{minipage}
\hfill
\begin{minipage}[t]{.48\textwidth}
	\vspace{0pt}
	\centering
	\includegraphics[width=\linewidth]{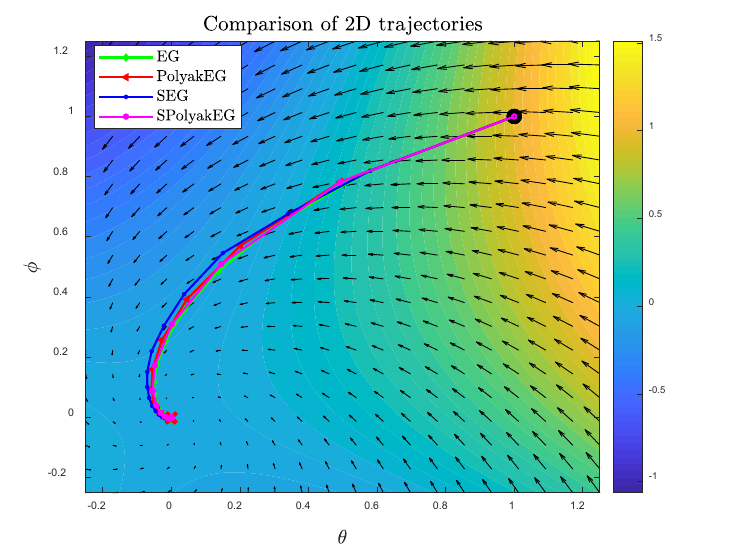}
	\captionsetup{font=small}
	\captionof{figure}{Trajectories of \algname{EG}, \algname{SEG}, \algname{PolyakEG} and \algname{SPolyakEG} for \eqref{eqn:cc-plot} using $\gamma_k=1/L$. The unique solution $(0,0)$ is marked in red.}
	\label{fig:cc2-3}
\end{minipage}

\begin{figure}[ht]
	\centering
	\begin{subfigure}[t]{.5\textwidth}
		\centering
		\includegraphics[width=\linewidth]{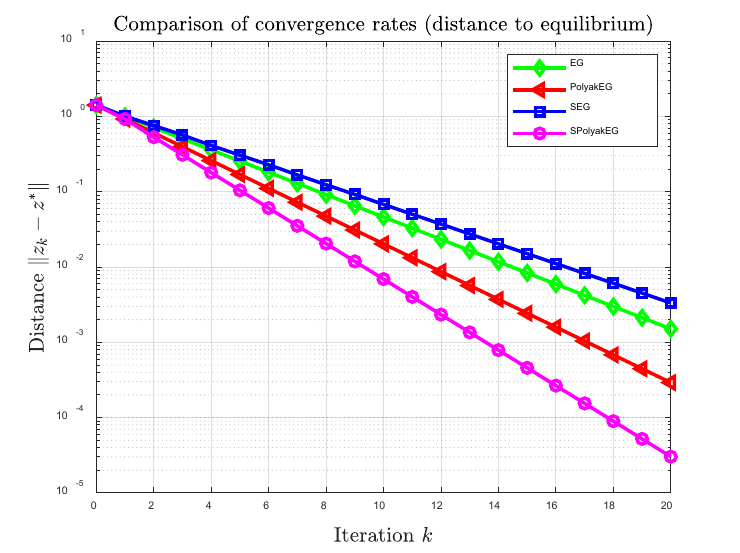}
		\caption{}
		\label{fig:cc2-1}
	\end{subfigure}
	\hfill
	\begin{subfigure}[t]{.5\textwidth}
		\centering
		\includegraphics[width=\linewidth]{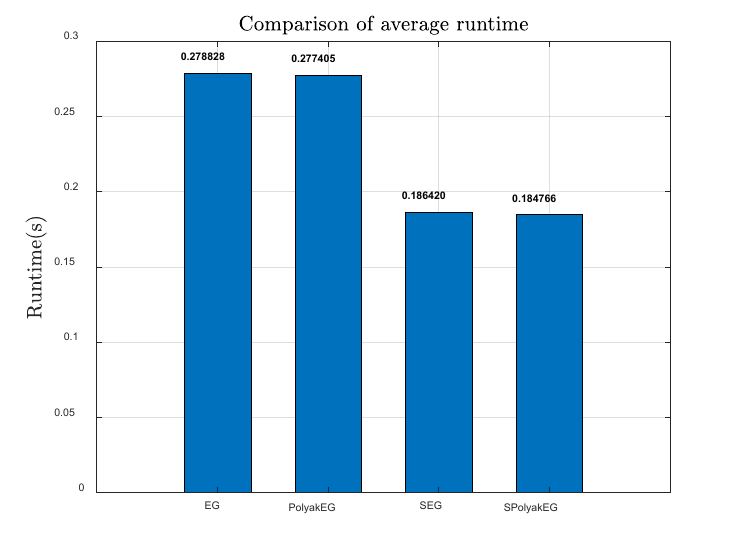}
		\caption{ }
		\label{fig:cc2-2}
	\end{subfigure}
	\caption{Panel (a) compares the convergence behavior of the four methods in terms of the relative distance to the saddle point.
		Panel (b) compares the runtime of the four methods. The single-call methods, \algname{SEG} and \algname{SPolyakEG}, require only one new operator evaluation per iteration, whereas \algname{EG} and \algname{PolyakEG} require two. 
	}
	\label{fig:cc}
\end{figure} 

\Cref{fig:cc} reports the numerical behavior of all four algorithms on the convex-concave minimax problem.
As can be seen from the relative distance curves in \Cref{fig:cc2-1}, \algname{SPolyakEG} exhibits the fastest convergence toward the unique saddle point, followed by \algname{SEG}, \algname{PolyakEG}, and \algname{EG}.
The Polyak-type methods exhibit faster iteration-wise convergence than their fixed-step counterparts, illustrating the benefit of the adaptive correction step size $\alpha_k$. 
In particular, despite relying on the recycled operator information from the preceding iteration, \algname{SPolyakEG} achieves faster per-iteration convergence than \algname{SEG}. This observation suggests that the adaptive Polyak correction can mitigate the effect of the delayed operator information.

In terms of computational efficiency, the runtime comparison in \Cref{fig:cc2-2} shows that the single-call methods, \algname{SEG} and \algname{SPolyakEG}, require less computational time than \algname{EG} and \algname{PolyakEG}. This reduction is consistent with their use of only one new operator evaluation per iteration.
Overall, these results demonstrate that SPolyakEG combines the single-call computational structure of \algname{SEG} with an adaptive Polyak correction, improving iteration-wise convergence while retaining the computational advantage of using only one new operator evaluation per iteration.

\subsection{The Dirac-GAN problem}

\begin{minipage}[t]{.50\textwidth}
	\vspace{0pt}
	We consider the Dirac-GAN problem \cite{fiez2021minimax}
	\begin{equation*}
		\underset{\theta \in \reals}{\min}\,
		\underset{\phi \in \reals}{\max}\,
		f(\theta,\phi)
		=
		-\log_2\left(1+\exp(-\theta\phi)\right)-1,
	\end{equation*}
	which is equivalent to \eqref{eq:root} with operator $
	F(\theta,\phi)
	=
	\frac{1}{\ln 2}
	\left(
	\phi\big(1-\sigma(\theta\phi)\big),
	-\theta\big(1-\sigma(\theta\phi)\big)
	\right)$,
	where $\sigma(t)=1/(1+e^{-t})$ denotes the sigmoid function.
	This problem admits the unique saddle point $(0,0)$.
	\Cref{fig:DiracGAN2} illustrates the trajectories of the first 50
	iterations of \algname{EG}, \algname{SEG}, \algname{PolyakEG}, and
	\algname{SPolyakEG}, initialized at
	$(\theta_0,\phi_0)=(1,1)$.
	All methods use the fixed extrapolation step size $\gamma_k=0.5$.
	We observe that the Polyak-type methods exhibit faster per-iteration
	convergence than their fixed-step counterparts.
	In particular, \algname{SPolyakEG} converges faster than
	\algname{SEG}, despite using only one new operator evaluation per
	iteration.
	This observation illustrates the benefit of the adaptive Polyak
	correction within the single-call framework.
\end{minipage}
\hfill
\begin{minipage}[t]{.48\textwidth}
	\vspace{0pt}
	\centering
	\includegraphics[width=\linewidth]{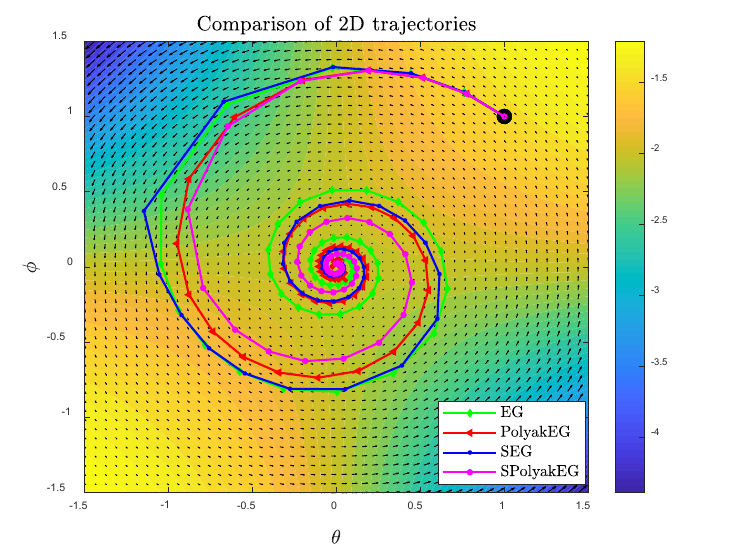}
	\captionsetup{font=small}
	\captionof{figure}{Trajectories of \algname{EG}, \algname{SEG},
		\algname{PolyakEG}, and \algname{SPolyakEG} for the Dirac-GAN problem
		using $\gamma_k=0.5$.
		The unique solution $(0,0)$ is marked in red.}
	\label{fig:DiracGAN2}
\end{minipage}

\begin{figure}[ht]
	\centering
	\begin{subfigure}[t]{.49\textwidth}
		\centering
		\includegraphics[width=\linewidth]{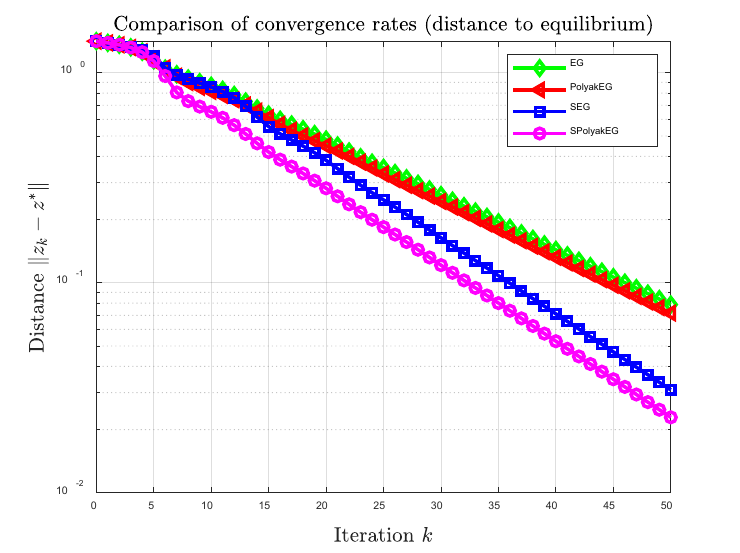}
		\caption{Relative distance to the saddle point versus iteration count.}
		\label{fig:DiracGAN2-1}
	\end{subfigure}
	\hfill
	\begin{subfigure}[t]{.49\textwidth}
		\centering
		\includegraphics[width=\linewidth]{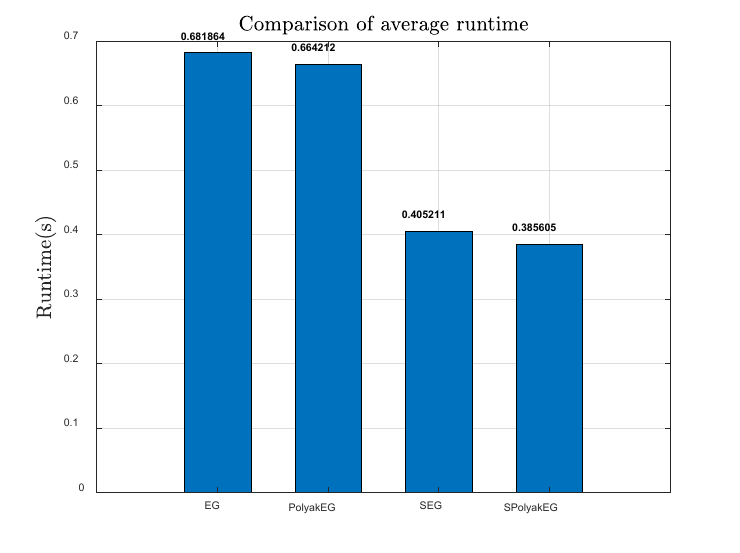}
		\caption{Average runtime over repeated trials.}
		\label{fig:DiracGAN2-2}
	\end{subfigure}
	\caption{Numerical results on the Dirac-GAN problem.
		Panel (a) compares the convergence behavior of the four methods
		in terms of the relative distance to the saddle point.
		Panel (b) compares their average computational runtime.
		The single-call methods require only one new operator evaluation
		per iteration.
	}
	\label{fig:DiracGAN2-full}
\end{figure}

\Cref{fig:DiracGAN2-full} presents the quantitative performance comparison
on the Dirac-GAN problem.
As shown in \Cref{fig:DiracGAN2-1}, \algname{SPolyakEG} exhibits the
fastest decay of the relative distance to the saddle point among the
four tested methods, followed by \algname{SEG}, \algname{PolyakEG}, and
\algname{EG}.
The Polyak-type methods exhibit faster iteration-wise convergence than
their fixed-step counterparts, illustrating the benefit of the adaptive
correction step size.

In particular, despite using the recycled operator information from the
preceding iteration, \algname{SPolyakEG} achieves faster iteration-wise
convergence than \algname{SEG}.
This observation suggests that the adaptive Polyak correction can mitigate
the effect of the delayed operator information while preserving the
single-call structure.

In terms of computational efficiency, \Cref{fig:DiracGAN2-2} shows that
the single-call methods, \algname{SEG} and \algname{SPolyakEG}, require
approximately $42\%$ less runtime than the two-evaluation methods,
\algname{EG} and \algname{PolyakEG}.
These results show that \algname{SPolyakEG} combines the single-call
computational structure with an adaptive Polyak correction, leading to
improved iteration-wise convergence while retaining the computational
advantage of one new operator evaluation per iteration.

\section{Conclusion}
\label{sec:conclusion}

We proposed SPolyakEG, a single-call variant of the Polyak
extragradient method for solving the root-finding problem $
F(x^\star)=0$.
The method recycles the previously computed extrapolation operator
evaluation, reducing the number of new operator evaluations from two
to one per iteration after initialization.

To analyze the effect of the delayed operator information, we
introduced a delayed critical condition and established the resulting
bounds on consecutive operator values and the adaptive Polyak
stepsize. These properties ensure the positivity of the adaptive
stepsize and lead to a Fej\'er-type descent inequality. Under
monotonicity, this yields an $O(1/K)$ convergence rate for the
squared residual evaluated at the extrapolation points. Under strong
monotonicity and a suitable stepsize condition, we further establish
linear convergence of the iterates. For Lipschitz continuous
operators, we also provide sufficient conditions on the
extrapolation stepsize under which the delayed critical condition is
satisfied.

From the computational perspective, the main advantage of SPolyakEG
is its reduced oracle complexity. PolyakEG requires two fresh
evaluations of the operator per iteration, whereas SPolyakEG requires
only one after the initialization step. Hence, for problems where
operator evaluations dominate the computational cost, the proposed
scheme provides a single-call alternative while maintaining the same
order of residual convergence and linear convergence under strong
monotonicity within the established theoretical conditions.

\section*{Acknowledgments}
This work was supported by the General Program of Chongqing Natural Science Foundation (Grant No. CSTB2025NSCQ-GPX0795) and the Scientific and Technological Research Program of Chongqing Municipal Education Commission (Grant No. KJQN202603101).

\bibliography{bibfile}

\end{document}